\documentclass[a4paper, 12pt]{article}

\usepackage[english,francais]{babel}
\usepackage[T1]{fontenc}
\usepackage{lmodern}
\usepackage{array}
\usepackage[hmargin=2.5cm,vmargin=30mm]{geometry}
\usepackage{amsthm,amssymb,amsmath}
\usepackage{wrapfig,subfig}
\usepackage{hyperref}

\usepackage{graphics}
\def\picinput#1{\includegraphics{#1.pdf}}

\newtheorem{theorem}{Theorem}
\newtheorem{proposition}{Proposition}
\newtheorem{lemma}{Lemma}
\newtheorem{corollary}{Corollary}

\newtheorem*{lemma*}{Lemma}

\def\N{\mathbb{N}}
\def\Z{\mathbb{Z}}
\def\Q{\mathbb{Q}}
\def\R{\mathbb{R}}
\def\C{\mathbb{C}}

\def\Deck{\operatorname{Deck}}
\def\SL{\operatorname{SL}}

\def\PSL{\operatorname{PSL}}

\def\Id{\operatorname{Id}}
\def\Re{\operatorname{Re}}
\def\Im{\operatorname{Im}}

\def\SO{\operatorname{SO}}

\def\HHH{\mathcal{H}}
\def\QQQ{\mathcal{Q}}
\def\LLL{\mathcal{L}}
\def\MMM{\mathcal{M}}

\def\Tsurf{\mathrm{T}}
\def\Xsurf{\mathrm{X}}
\def\Lsurf{\mathrm{L}}

\def\Xablocus{\mathcal{G}}

\def\epsilon{\varepsilon}

\title{Diffusion for the periodic wind-tree model}
\author{Vincent Delecroix, Pascal Hubert, Samuel Leli\`evre}
\date{}

\begin{document}
\maketitle

\let\thefootnote\relax\footnotetext{\textit{2000 Mathematics Subject Classification class:} 30F30, 37E35, 37A40.}
\let\thefootnote\relax

\selectlanguage{english}

\begin{abstract}
The periodic wind-tree model is an infinite billiard in the plane with identical rectangular scatterers placed at each integer point. We prove that independtly of the size of scatters and generically with respect to the angle, the polynomial diffusion rate in this billiard is $2/3$.

\vskip 0.5\baselineskip

\selectlanguage{francais}
\begin{center}\textbf{R\'esum\'e}\end{center}
\vskip 0.5\baselineskip
{\bf Diffusion du vent dans les arbres}

\vspace{0.5\baselineskip}

Le vent dans les arbres p\'eriodique est un billard infini construit de la mani\`ere suivante. On consid\`ere le plan dans lequel sont plac\'es des obstacles rectangulaires identiques \`a chaque point entier. Une particule (identifi\'ee \`a un point) se d\'eplace en ligne droite (le vent) et rebondit de mani\`ere \'elastique sur les obstacles (les arbres). Nous prouvons qu'ind\'ependamment de la taille des obstacles et g\'en\'eriquement par rapport \`a l'angle initial de la particule le coefficient de diffusion polynomial des orbites de ce billard est $2/3$.
\end{abstract}

\selectlanguage{english}

\section{Introduction}

\label{section:introduction}
The wind-tree model is a billiard in the plane introduced by P.~Ehrenfest and T.~Ehrenfest in 1912 (\cite{EhEh}). We study the periodic version studied by J.~Hardy and J.~Weber~\cite{HaWe}. A point moves in the plane $\R^2$ and bounces elastically off rectangular scatterers following the usual law of reflection. The scatterers are translates of the rectangle $[0,a]\times[0,b]$ where $0<a<1$ and $0<b<1$, one centered at each point of $\Z^2$. We denote the complement of obstacles in the plane by $\Tsurf(a,b)$ and refer to it as the \emph{wind-tree model} or the \emph{infinite billiard table}. Our aim is to understand dynamical properties of the wind-tree model. We denote by $\phi^\theta_t: \Tsurf(a,b) \rightarrow \Tsurf(a,b)$ the billiard flow: for a point $p \in \Tsurf(a,b)$, the point $\phi^\theta_t(p)$ is the position of a particle after time $t$ starting from position $p$ in direction $\theta$.

It is proved in~\cite{HaWe} that the rate of diffusion in the periodic wind-tree model is $\log t \log \log t$ for very specific directions (generalized diagonals which corresponds to angles of the form $\arctan(p/q)$ with $p/q \in \Q$). Their result was recently completed by J.-P.~Conze and E.~Gutkin~\cite{CG} who explicit the ergodic decomposition of the billiard flow for those directions. K.~Fr\c{a}czek and C.~Ulcigrai recently proved that generically the billiard flow is non-ergodic. P.~Hubert, S.~Leli\`evre and S.~Troubetzkoy~\cite{HLT} proved that for a residual set of parameters $a$ and $b$, for almost every direction $\theta$, the flow in direction $\theta$ is recurrent. In this paper, we compute the polynomial rate of diffusion of the orbits which is valid for almost every direction $\theta$. We get the following. 
\begin{theorem} \label{thm:divergence_rate_billiard}
Let $d(.,.)$ be the Euclidean distance on $\R^2$. Then for all parameters $(a,b) \in (0,1)^2$, Lebesgue-almost all $\theta$ and every point $p$ in $\Tsurf(a,b)$ (with an infinite forward orbit)
\[
\limsup_{T \to +\infty} \frac{\log d(p,\phi^\theta_T(p))}{\log T} = \frac{2}{3}.
\]
\end{theorem}
By the $\Z^2$-periodicity of the billiard table $\Tsurf(a,b)$, our problem reduces to understand deviations of a-$\Z^2$ cocycle over the billiard in a fundamental domain. On the other hand, as the barriers are horizontals and verticals, an orbit in $\Tsurf(a,b)$ with initial angle $\theta$ from the horizontal takes at most four different directions $\{\theta,\pi-\theta,-\theta,\pi+\theta\}$ (the billiard is rational). By a standard construction consisting of unfolding the trajectories~\cite{Ta}, called the Katok-Zemliakov construction, the billiard flow can be replaced by a linear flow on a (non compact) translation surface which is made of four copies of $\Tsurf(a,b)$ that we denote $\Xsurf_\infty(a,b)$ (see section~\ref{section:Xinfty_as_a_Z2_cover} for the construction). The surface $\Xsurf_\infty(a,b)$ is $\Z^2$-periodic and we denote $\Xsurf(a,b)$ the quotient of $\Xsurf_\infty(a,b)$ under the $\Z^2$ action. As the unfolding procedure of the billiard flow is equivariant with respect to the $\Z^2$ action, $\Xsurf(a,b)$ can be also be seen as the unfolding of the billiard in a fundamental domain of the action of $\Z^2$ on the billiard table $\Tsurf(a,b)$.

The position of the particle in $\Xsurf_\infty(a,b)$ can be tracked from $\Xsurf(a,b)$. More precisely, the position of the particle starting from $p \in \Xsurf_\infty(a,b)$ in direction $\theta$ can be approximated by the pairing of a geodesic $\gamma_t(p)$ of $\Xsurf(a,b)$ seen as an element of the homology with a cocycle $f \in H^1(\Xsurf(a,b); \Z^2)$ describing the infinite cover $\Xsurf_\infty(a,b) \rightarrow \Xsurf(a,b)$. The growth of pairing of a fixed cocycle with geodesics in a translation surface is equivalent to the growth of certain Birkhoff sums over an interval exchange transformation. The estimation can be obtained from the action of $\SL(2,\R)$ on strata of translation surfaces $\HHH_g(\alpha)$ and more precisely of the Teichm\"uller flow which corresponds to the action of diagonal matrices $g_t = \left( \begin{smallmatrix}e^t&0\\0&e^{-t}\end{smallmatrix}\right)$ (see Section~\ref{section:background} for precise definitions). As proved by A.~Zorich~\cite{Zo1,Zo2} the Kontsevich-Zorich cocycle over the Teichm\"uller flow can be used to estimate the deviations of Birkhoff sums for generic interval exchange transformations with respect to the Lebesgue measure. More precisely, he proved that the Lyapunov exponents of the Kontsevich-Zorich cocycle is the polynomial rate of deviations. G.~Forni~\cite{Fo} relates this phenomenon to obstructions to solve cohomological equations and extends Zorich's proof to a more general context (see section~9 of~\cite{Fo}).

The surface $\Xsurf(a,b)$ is a covering of the genus $2$ surface $\Lsurf(a,b)$ which is a so called L-shaped surface that belongs to the stratum $\HHH(2)$. The orbit of $\Xsurf(a,b)$ for the Teichm\"uller flow belongs to a sub-locus of the moduli space $\HHH(2^4)$ that we call $\Xablocus$.

We now formulate a generalization of A.~Zorich's and G.~Forni's theorems about deviations of ergodic averages that is a central step in the proof of Theorem~\ref{thm:divergence_rate_billiard}. Let $\HHH(\alpha)$ be a stratum of Abelian differentials and $Y \in \HHH(\alpha)$ a translation surface. The Teichm\"uller flow $(g_t)$ can be used to renormalize the trajectories of the linear flow on $Y$. The Kontsevich-Zorich cocycle $B^{(t)}(Y):H^1(Y;\R) \rightarrow H^1(g_t \cdot Y;\R)$ (or KZ cocycle) measures the growth of cohomology vectors along the Teichm\"uller geodesic $\left(g_t \cdot Y \right)_t$. Let $\mu$ be a $g_t$-invariant ergodic probability measure on $\HHH(\alpha)$. It follows from \cite{Fo}, that the KZ cocycle is integrable for the measure $\mu$. From Oseledets multiplicative ergodic theorem, there exists real numbers $\nu_1(\mu) > \nu_2(\mu) > \ldots > \nu_k(\mu) > 0$, such that for $\mu$-almost every non zero Abelian differential $Y \in \HHH(\alpha)$ there exists a unique flag
\begin{align*}
H^1(Y;\R) = F^u_1 \supset F^u_2 \supset \ldots &\supset F^u_k \supset F^u_{k+1} = F^c \supset F^s_k \supset \ldots \supset F^s_1 \supset F^s_0 = \{0\}
\end{align*}
such that for any norm $\|.\|$ on $H^1(Y;\R)$, for all $1 \leq i \leq k$
\begin{enumerate}
\item if $f \in F^u_i \backslash F^u_{i+1}$, then
\[\lim_{t \to \infty} \frac{\log \|B^{(t)}(Y) \cdot f\|}{\log t} = \nu_i(\mu),\]
\item if $f \in F^s_i \backslash F^s_{i-1}$, then
\[\lim_{t \to \infty} \frac{\log \|B^{(t)}(Y) \cdot f\|}{\log t} = -\nu_i(\mu),\]
\item if $f \in F^c \backslash F^s_k$, then
\[\lim_{t \to \infty} \frac{\log \|B^{(t)}(Y) \cdot f\|}{\log t} = 0.\]
\end{enumerate}
There exists also positive integers $m_i$ for $i=1,\ldots,k$ and an integer $m$ such that for $\mu$ almost all translation surface $Y$ the filtration satisfies
\begin{itemize}
\item the dimension of $F^s_i$ is $m_1 + \ldots + m_i$,
\item the dimension of $F^c$ is $m_1 + \ldots + m_k + 2m$,
\item the dimension of $F^u_i$ is $m_1 + \ldots + m_{i-1} + 2m_i + \ldots + 2m_k + 2m$.
\end{itemize}
From the definition of the Teichm\"uller flow and the KZ cocycle, it follows that $\nu_1 = 1$. Forni proved that $m_1 = 1$ \cite{Fo}. The \emph{Lyapunov spectrum of the KZ} cocycle is the multiset of numbers
\[
\begin{array}{ccccccccc}
\nu_1 = 1 & \underbrace{\nu_2 \ldots \nu_2} & \ldots & \underbrace{\nu_k \ldots \nu_k} & \underbrace{0 \ldots 0} & \underbrace{-\nu_k \ldots -\nu_k} & \ldots & \underbrace{-\nu_2 \ldots -\nu_2} & -1 = -\nu_1 \\
 & \text{$m_2$ times} & \ldots & \text{$m_k$ times} & \text{$2m$ times} & \text{$m_k$ times}  & \ldots & \text{$m_2$ times}  \\
\end{array}
\]
The numbers $\nu_i(\mu)$ for $i=1,\ldots,k$ are called the \emph{positive Lyapunov exponents (with respect to $\mu$)}. The subspace $F^s = F^s_k$ is called the \emph{stable space} (at $Y$) of the KZ cocycle.

In order to state a precise statement for deviations, one needs genericity with respect to Lyapunov exponents but also an extra assumption on recurrence. Let $\mu$ be a $g_t$ ergodic measure on some stratum $\HHH(\alpha)$. We say that a surface $Y \in \HHH(\alpha)$ is \emph{generic recurrent for $\mu$} if there exist compact neighborhoods $U_i \subset \HHH(\alpha)$ of $Y$ such that $\bigcap_i U_i = \{Y\}$ and
\[
\lim_{t \to \infty} \frac{Leb(\{s; s \in [0,t] \quad \text{and} \quad g_s Y \in U\})}{t} = \mu(U).
\]
Birkhoff theorem ensures that this conditions is satisfied for almost every surface.
\begin{theorem} \label{thm:deviations}
Let $\mu$ be a $g_t$-ergodic measure on a stratum of Abelian differentials. Let $\nu_i$ for $i=1,\ldots,k$ denotes the positive Lyapunov exponents of the KZ cocycle for $\mu$ and denote, for an Oseledets generic surface $Y$, $F^u_i(Y)$, $F^c(Y)$ and $F^s_i(Y)$ the components of the flag of the Oseledets decomposition.

Then, for a surface $Y \in \HHH(\alpha)$ which is generically recurrent and Oseledets generic for $\mu$, for every point $p \in Y$ with an infinite forward orbit
\begin{enumerate}
\item along the unstable space the growth is polynomial: for all $1 \leq i \leq k$, for all $f \in F^u_i \backslash F^u_{i+1}$ \label{item:deviations_pos}
\[ \limsup_{T \to \infty} \frac{\log |\langle f, \gamma_T(p) \rangle |}{\log T} = \nu_i,\]
\item along the central space the growth is sub-polynomial: for all $f \in F^c \backslash F^s_k$
\[\limsup_{T \to \infty} \frac{\log |\langle f, \gamma_T(p) \rangle |}{\log T} = 0,\]
\item along the stable space the growth is bounded: there exists a constant $C$ such that for all $f \in F^s$
\[\forall T \geq 0,\ |\langle f, \gamma_T(p) \rangle | \leq C \|f\|.\]
\end{enumerate}
\end{theorem}
Theorem~\ref{thm:deviations} has first been proved by A.~Zorich~\cite{Zo0,Zo1,Zo2} for the Lebesgue measure on a connected component of a stratum or equivalently for a generic interval exchange transformation. G.~Forni~\cite{Fo} extended the theorem for a very large class of functions and for certain measures. More precisely, his proof of the lower bound relies on the existence of a particular translation surface in the support of the measure. A.~Bufetov~\cite{Bu} gave a proof of case~\ref{item:deviations_pos} of Theorem~\ref{thm:deviations} (when the cocycle $f$ is associated with a positive Lyapunov exponent) in the general context of symbolic dynamics which applies in particular to translation flows (Proposition 2. and 5. of \cite{Bu}). Our approach uses Veech's zippered rectangles~\cite{Ve1} and gives a concrete version of the renormalization process by the Teichm\"uller flow and the Kontsevich-Zorich cocycle in the flavor of~\cite{Zo1,Zo2} and~\cite{Fo}.

On the other hand, from results of A.~Eskin, M.~Kontsevich and A.~Zorich~\cite{EKZ2} about sum of Lyapunov exponents in hyperelliptic loci, we deduce that the Lyapunov exponent for $\Xsurf(a,b)$ which controls the deviation in the wind-tree model equals $2/3$. The value $2/3$ comes from algebraic geometry. More precisely, it corresponds to the degree of a subbundle of the Hodge bundle over the moduli space of complex curves (or Riemann surfaces) in which belongs the wind-tree cocycle.

Using only Birkhoff and Oseledets theorem, one can prove that the conclusion of Theorem~\ref{thm:divergence_rate_billiard} holds for almost every parameters $a,b$. In order to obtain all parameters we use a recent result of J.~Chaika and A.~Eskin~\cite{CE} which asserts that Birkhoff theorem for regular functions and Oseledets theorem for the Kontsevich-Zorich cocycle are more regular for $\SL(2,\R)$-invariant measures: they hold for all surfaces in almost every directions. The work of Chaika and Eskin strongly relies on previous work of A. Eskin and M. Mirzhakani~\cite{EM} and A. Eskin, M. Mirzakhani and M. Mohamadi~\cite{EMM} on $\SL(2,\R)$-invariant measures on strata of Abelian differentials.

The paper is organised as follows. In Section~\ref{section:background} we introduce the tools from Teichm\"uller theory which are involved in our proof of Theorem~\ref{thm:divergence_rate_billiard}. In Section~\ref{section:from_infinite_billiard_to_finite_surface}, we detail the unfolding procedure and prove that the distance in Theorem~\ref{thm:divergence_rate_billiard} corresponds to a pairing between a geodesic in $\Xsurf(a,b)$ with an integer cocycle. Then we reformulate Theorem~\ref{thm:divergence_rate_billiard} in the language of translation surfaces (see Theorem~\ref{thm:divergence_rate_surface}). In Section~\ref{sec:computation_of_Lyapunov_exponents} we compute the Lyapunov exponents relative to every measure on $\HHH(2^4)$ which is supported on the closure of the $\SL(2,\R)$-orbit of a surfaces $\Xsurf(a,b)$. Section~\ref{section:deviations} is devoted to the proof of Theorem~\ref{thm:deviations}.

\bigskip

The preprint~\cite{CE} appeared after preliminary versions of this paper. In earlier versions, conclusion of Theorem~\ref{thm:divergence_rate_billiard} was weaker and we rely heavily on classification of $\SL(2,\R)$-invariant measures in genus $2$ by K.~Calta~\cite{Ca} and McMullen~\cite{Mc1,Mc2,Mc3}.

\bigskip

\textbf{Acknowledgments:} The authors heartily thank A.~Avila, A.~Bufetov, G.~Forni and A.~Zorich for very fruitful discussions.

\tableofcontents

\section{Background} \label{section:background}
The main objects in this paper are:
\begin{itemize}
\item
\label{compact-transl-surf}
closed compact translation surfaces -- equivalently, closed compact Riemann surfaces endowed with a holomorphic 1-form;
\item
\label{infinite-area-transl-surf}
infinite-area periodic translation surfaces.
\end{itemize}
For general references on translation surfaces and interval exchange transformations we refer the reader to the survey of A.~Zorich~\cite{Zo3}, J.-C.~Yoccoz~\cite{Yo} or the notes of M.~Viana~\cite{Vi}.

A \emph{translation surface} is a surface which can be obtained by edge-to-edge gluing of polygons in the plane using translations only. Such a surface is endowed with a flat metric (the one from $\R^2$) and a canonical directions. There is a one to one correspondence between compact translation surfaces and compact Riemann surfaces equipped with a non-zero holomorphic 1--form. If $(Y,\omega)$ is a Riemann surface together with a holomorphic one-form, the flat metric corresponds to $|\omega|^2$. In particular, the area of $(Y,\omega)$ is $i/2 \int \omega \wedge \overline{\omega}$.

In a translation surface, directions are globally defined. Hence the geodesic flow in a direction can be defined on the surface. There is a canonical vertical direction in each translation surface and we refer to the flow in this direction as the \emph{linear flow}. The flow in the direction $\theta \in [0,2\pi)$ for the differential $\omega$ on $Y$ is the linear flow of $e^{-i\theta} \omega$ on $Y$. Note that the flow is not defined at the zeros of $\omega$.

Now we define the moduli space of translation surfaces. We use a marking in order to avoid symmetries which create singularities in the moduli space. Let $\alpha = (\alpha_1,\ldots,\alpha_s)$ and $g$ be integers such that $\alpha_1 + \ldots + \alpha_s = 2g-2$. Let $S$ be a compact (toplogical surface) of genus $g$ and let $\Sigma = \{x_1, x_2, \ldots, x_s\}$ be a set of $s$ points in $S$.
The stratum $\HHH_g(\alpha)$ is the set of (equivalence classes) translation structure on $S$ with
\begin{enumerate}
  \item zeros of degree $\alpha_i$ at $x_i$, and regular out of $\Sigma$,
  \item an horizontal separatrix is fixed at each $x_i$.
\end{enumerate}
Two translation structure $\omega$ and $\eta$ on $S$ are identified if there exists a diffeomorphism $\phi: S \rightarrow S$ that fix pointwise the set $\Sigma$ maps $\omega$ to $\eta$ and maps the marked horizontal separatrix of $\omega$ at $x_i$ to the marked horizontal separatrix of $\eta$ at $x_i$. We often use exponential notation for $\alpha$, for example $\HHH(2^4)$ means $\HHH(2,2,2,2)$ in our context. These strata can have up to three connected components, which were classified by M.~Kontsevich and A.~Zorich~\cite{KZ}, and distinguished by two invariants: hyperellipticity and parity of spin structure. We denote by $\HHH^{(1)}(\alpha) \subset \HHH(\alpha)$ the codimension $1$ subspace which consists of area $1$ translation surfaces.

Each stratum $\HHH_g(\alpha)$ carries a natural affine structure which makes it a manifold (if we forget markings, we obtain an orbifold). The affine structure is obtained from the association of $(S,\omega) \in \HHH_g(\alpha) \mapsto [\omega] \in H^1(S, \Sigma; \C)$ (the period map). On translation surfaces obtained by polgon gluings, this map may be seen as the edges (as element of $\C$). In that model, two surfaces are nearby if they are obtained from the same gluings and the edges are nearby for the natural topology on finite dimensional vector spaces.

There is a natural action of $\SL_2(\R)$ on components of strata $\HHH(\alpha)$ coming from the linear action of $\SL(2,\R)$ on $\R^2$. More precisely, let $(Y,\omega)$ be a translation surface obtained by gluing a finite family of polygons $(P_i)$ and $g \in \SL_2(\R)$. Then the surface $g \cdot (Y,\omega)$ is the surface obtained by gluing the polygons $(g \cdot P_i)$. The \emph{Teichm\"uller geodesic flow} on $\HHH_g$ is the action of the diagonal matrices $\displaystyle g_t = \begin{pmatrix} e^t&0\\0&e^{-t}\end{pmatrix}$. The image of the orbits $(g_t \cdot (X,\omega))_t$ in $\MMM_g$ are geodesic with respect to the Teichm\"uller metric. Each stratum $\HHH_g(\alpha)$ carries a natural \emph{Lebesgue measure}, invariant under the action of $\SL(2,\R)$. Moreover, this action preserves the area and hence $\HHH^{(1)}(\alpha)$. H.~Masur~\cite{Ma} and independently W.~Veech~\cite{Ve1} proved that on each component of a normalised stratum $\HHH^{(1)}(\alpha)$ the total mass of the Lebesgue measure is finite and the geodesic flow acts ergodically with respect to this measure. Another important one parameter flow on $\HHH(\alpha)$ is the \emph{horocycle flow} given by the action of $\displaystyle h_s = \begin{pmatrix}1&s\\0&1\end{pmatrix}$.

More generally, one can consider the strata of quadratic differentials with at most simple poles $\QQQ_g(\alpha)$ where $\alpha$ is an integer partition of $4g-4$. The degree $\alpha_i$ corresponds to a conic point of angle $(2+\alpha_i)\pi$. A translation surface associated to a quadratic differential may has non trivial holonomy with value in $\{1,-1\}$. The action of $\SL(2,\R)$ on Abelian differentials extends to quadratic differentials.

Stabilisers for the action of $\SL(2,\R)$ on $\HHH_g$ or $\QQQ_g$, called \emph{Veech groups}, are discrete non-cocompact subgroups of $\SL_2(\R)$; they are trivial (i.e. either $\{\Id\}$ or $\{\Id,-\Id\}$) for almost every surface in each stratum component, and in exceptional cases are lattices (i.e.\ finite-covolume subgroups) in $\SL_2(\R)$. In such cases, the surface satisfies the Veech dichotomy: in every direction, the linear flow is either uniquely ergodic, or decomposes the surface into a finite union of cylinders of periodic trajectories (see~\cite{Ve1}). Closed compact translation surfaces with a lattice Veech group are exactly those whose $\SL_2(\R)$-orbit is closed in the corresponding stratum component. They are called \emph{Veech surfaces}. Their orbits project to \emph{Teichm\"uller curves} in the moduli space $\MMM_g$ of closed compact Riemann surfaces of genus $g$. A translation surface is a \emph{square-tiled surface} if it is a ramified cover of the torus $\R^2/\Z^2$ with only $0$ as ramification point. Square-tiled surfaces are examples of Veech surfaces. Their Veech groups are commensurable to $\SL_2(\Z)$.

The simplest stratum besides the one of tori is $\HHH(2)$ which consists of equivalence classes of 1--forms with a double zero (in flat surfaces terms a cone point of angle $6\pi$) on Riemann surfaces of genus two. Important examples of such surfaces are given by the family of surfaces $\Lsurf(a,b)$ with $0<a<1,0<b<1$ which are built as follows (see also Figure~\ref{fig:Lab}). Let $0<a<1$ and $0<b<1$. Consider the polygon with extremal points $(0,0)$, $(1-a,0)$, $(1,0)$, $(1,1-b)$, $(1-a,1-b)$, $(1-a,1)$, $(0,1)$, $(0,1-b)$ and glue the opposite sides together:
\begin{enumerate}
\item $[(0,0),(1-a,0)]$ with $[(0,1),(1-a,1)]$ (the side $h_1$ labeled on Figure~\ref{fig:Lab}),
\item $[(1-a,0),(1,0)]$ with $[(1-a,1-b),(1,1-b)]$ (the side $h_2$),
\item $[(0,0),(0,1-b)]$ with $[(1,0),(1,1-b)]$ (the side $v_1$),
\item $[(0,1-b),(0,1)]$ with $[(1-a,1-b),(1-a,1)]$ (the side $v_2$).
\end{enumerate}

\begin{figure}[!ht]
\begin{center}\picinput{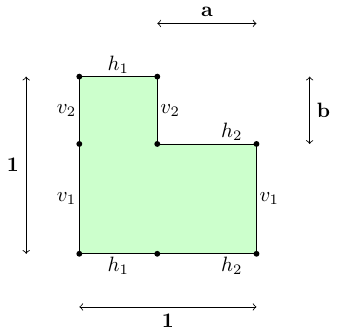}\end{center}
\caption{The surface $\Lsurf(a,b)$ built from a L-shaped polygon.}
\label{fig:Lab}
\end{figure}
The stratum $\HHH(2)$ is connected and is the best understood. It was proven that the Teichm\"uller curves are generated by surfaces of the form $\Lsurf(a,b)$.
\begin{theorem}[Calta~\cite{Ca}, McMullen~\cite{Mc1,Mc2}]
The surface $\Lsurf(a,b)$ is a Veech surface if and only if
\begin{enumerate}
\item either $a,b \in \Q$ in which case $\Lsurf(a,b)$ is square-tiled,
\item or there exists $x,y \in \Q$ and $D > 1$ a square-free integer such that $1/(1-a) = x + y \sqrt{D}$ and $1/(1-b) = (1-x) + y \sqrt{D}$.
\end{enumerate}
Moreover, any Teichm\"uller curve in $\HHH(2)$ contains (up to rescaling the area) a surface of the form $\Lsurf(a,b)$.
\end{theorem}
In his fundamental work, C.~McMullen~\cite{Mc3} proved a complete classification theorem for $\SL_2(\R)$-invariant measures and closed invariant set.
\begin{theorem}[McMullen, \cite{Mc3} Theorems~10.1 and 10.2 p. 440--441] \label{thm:McMullen_classification}
The only $\SL(2,\R)$-invariant irreducible closed subsets of $\HHH(2)$ are the Teichm\"uller curves and the whole stratum. The only $\SL(2,\R)$-invariant probability measures are the Haar measure carried on Teichm\"uller curves and the Lebesgue measure on the stratum.
\end{theorem}

Let $g \geq 2$ and $\alpha = (\alpha_1, \alpha_2, \ldots, \alpha_s)$ an integer partition of $2g-2$. The \emph{Hodge bundle} $E_g$ is the real vector bundle of dimension $2g$ over $\HHH_g(\alpha)$ where the fiber over $(S,\omega) \in \HHH_g(\alpha)$ is the real cohomology $H^1(S;\R)$.  Each fibre $H^1(X;\R)$ has a natural lattice $H^1(X;\Z)$ which allows identification of nearby fibers and definition of the Gauss-Manin (flat) connection. The holonomy along the Teichm\"uller geodesic flow provides a symplectic cocycle called the \emph{Kontsevich-Zorich cocycle}. It is formally defined as a map $B^{(t)}: H^1(X;\R) \rightarrow H^1(g_t \cdot X; \R)$. As each holonomy element corresponds to the action in homology of an element of the mapping class group of $X$, the map $B^{(t)}$ is symplectic. For a small transversal $U$ of the Teichm\"uler flow on $\MMM_g$ for which there exists a trivialization of the Hodge bundle, we identify all fibers with a fixed $H^1(X;\R)$ where $X \in U$. The sequence of first return times $0 = t_0 < t_1 < \ldots$ of $g_t \cdot X$ in $U$ gives a sequence of symplectic matrices in $H^1(X;\R)$ that we still denote $B^{(t_n)}(X)$.

For each $g_t$-invariant ergodic probability measure for the Teichm\"uller geodesic flow on $\HHH_g$, this cocycle has associated Lyapunov exponents. Based on computer experimentations, M.~Kontsevich~\cite{KZ0} conjectured a formula for the sum of positive Lyapunov exponents of the cocycle for Lebesgue measures on strata as well as for Veech surfaces. These formulas are now fully proven~\cite{EKZ1, EKZ2}.

An \emph{automorphism} of a translation surface $(S,\omega)$ is a diffeomorphism $\phi: S \rightarrow S$ that preserves $\omega$ (in other words, it acts by translations in the natural charts of $\omega$). We warn the reader that with our convention of markings, even if a point $(S,\omega)$ in $\HHH_g(\alpha)$ admits some automorphisms, it has trivial stabilizer as element of $\HHH_g(\alpha)$. In some concrete situations, as the one of the wind-tree model described in that article, the existence of automorphisms provides an $\SL_2(\R)$-equivariant splitting of the Hodge bundle. Under suitable assumptions for the $\SL(2,\R)$-subbundles (relative to variations of Hodge structure), it appears that for each of them there is a formula for the sum of positive Lyapunov exponents of the restricted Kontsevich-Zorich cocycle. Sometimes even individual Lyapunov exponents can be computed (see~\cite{BM}, \cite{FMZ}, \cite{EKZ1}). We recall a theorem of~\cite{EKZ2} which is a formula for the sum of Lyapunov exponents for the so called hyperelliptic locii of a stratum.

Quadratic differentials with at most simple poles on a Riemann surface are natural generalizations of translation surfaces. In that case, the holonomy is not necessarily trivial and may has values in $\{+1,-1\}$. On a Riemann surface with a quadratic differentials, directions are still globally defined and there is an action of $\PSL(2,\R)$. The stratum $\QQQ(d_1,\ldots,d_n)$ denotes the moduli space of area $1$ quadratic differentials with singularities of angles $(2+d_1)\pi, \ldots, (2+d_n)\pi$ which are not squares of a Abelian differentials.

Let $q$ be a quadratic differential on some Riemann surface $S$. The foliation on $S$ in some direction $\theta$ is not orientable. There is a canonical way to define a double cover $\pi: \tilde{S} \rightarrow S$ ramified at the singularities for which the degree $d_i$ is odd and for which $\pi^* q$ is the square of an Abelian differential on $\tilde{S}$. The locus of such double covers when the pair $(S,q)$ moves in its stratum gives a $\SL(2,\R)$ invariant locus in an Abelian stratum called \emph{orientation cover locus}. When, $S$ is a sphere, the orientation cover locus is called an \emph{hyperelliptic locus}. For Lyapunov exponents of hyperelliptic locii, the following general theorem holds.
\begin{theorem}[Eskin-Kontsevich-Zorich~\cite{EKZ2}, Corollary~1 p.~14] \label{thm:exp_lyap_hyp}
Let $\mu$ be an $\SL(2,\R)$-invariant ergodic probability measure on a stratum $\HHH_g(\alpha)$ of Abelian differential. Assume that $\mu$ comes from the orientation covering morphism of a $\SL(2,\R)$-invariant (regular) measure $\overline{\mu}$ on a stratum of quadratic differentials on the sphere $\QQQ(d_1,d_2,\ldots,d_n)$. Then, the sum of positive Lyapunov exponents $\nu_1 \geq \ldots \geq \nu_g$ for the measure $\mu$ is given by
\[
\nu_1 + \ldots + \nu_g = \frac{1}{4} \ \sum_{\text{j with $d_j$ odd}} \frac{1}{d_j+2}.
\]
\end{theorem}
In particular the value of the sum does not depend on the measure but only on the stratum $\QQQ(d_1,d_2,\ldots,d_n)$. For the condition of \emph{regular measure} which appears in the statement of Theorem~\ref{thm:exp_lyap_hyp} we refer to Definition 1 p. 9 of~\cite{EKZ2}. We emphasise that all known $\SL(2,\R)$-ergodic measures on strata of Abelian differentials are regular.


For infinite-area translation surfaces, it is not clear what the good notions of moduli spaces are. However, the action of $\SL_2(\R)$ still makes sense, and Veech groups can be defined~\cite{Va1,Va2}. An \emph{infinite periodic translation surface} is an infinite area translation surface which is an infinite normal cover of a finite area translation surface. We say $\Gamma$-infinite translation surface to specify the Deck group $\Gamma$. Examples of $\Z$-infinite translation surfaces are studied by P.~Hubert and G.~Schmith\"usen in~\cite{HS} and a general formalism is introduced by P.~Hooper and B.~Weiss in~\cite{HoWe}. For some particularly symmetric examples, it is possible to get a very complete picture of the dynamics~\cite{HoHuWe}. The family of surfaces $\Xsurf_\infty(a,b)$ obtained by unfolding the billiard tables $T(a,b)$ are $\Z^2$-infinite translation surfaces.

\section{From infinite billiard table to finite surface} \label{section:from_infinite_billiard_to_finite_surface}
First of all, the flow in the billiard table $\Tsurf(a,b)$ is invariant under $\Z^2$ translation. Secondly, the angles between the scatterers are multiples of $\pi/2$ and the Katok-Zemliakov construction conjugates the billiard flow on $T(a,b)$ to a linear flow on an infinite translation surface $\Xsurf_\infty(a,b)$. Using these two ingredients, we reduce the study of the billiard flow into the study of a $\Z^2$-cocycle over the liner flow of a finite translation surface $\Xsurf(a,b)$. The surface $\Xsurf(a,b)$ obtained by unfolding a fundamental domain of the table $\Tsurf(a,b)$ is an intermediate cover between the finite surface $\Lsurf(a,b)$ of genus $2$ and the infinite surface $\Xsurf_\infty(a,b)$. The surface $\Xsurf(a,b)$ is the main actor of this paper.

\textbf{Notation:} For the whole section, we fix $0 < a < 1$ and $0 < b <1$.

\subsection{Unfolding the fundamental domain} \label{subsection:unfolding_the_fundamental_domain}
A fundamental domain for the $\Z^2$ action on the infinite billiard $T(a,b)$ can be seen either as a torus with a square obstacle inside (see Figure~\ref{fig:fundamental_domain_torus}) or as a surface $\Lsurf = \Lsurf(a,b)$ with barriers on its boundary (see Figure~\ref{fig:fundamental_domain_L}).
\begin{figure}[!ht]
\centering
\subfloat[Fundamental domain of the billiard table $\Tsurf(a,b)$ as a torus with a rectangle scatterer.]{%
\label{fig:fundamental_domain_torus}%
\picinput{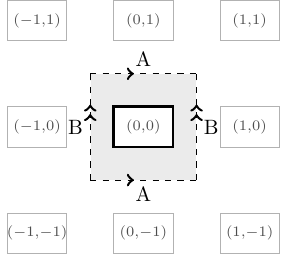}%
}
\hspace{2cm}
\subfloat[Fundamental domain of the billiard table $\Tsurf(a,b)$ as a L shaped surface with barriers.]{%
\label{fig:fundamental_domain_L}%
\picinput{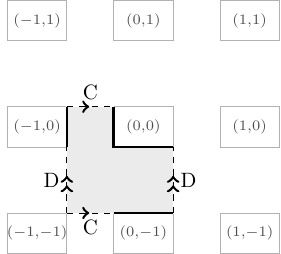}%
}
\caption{Two versions of the fundamental domains for the billiard table $\Tsurf(a,b)$. The boundaries of the scatterers are thick and the arrows together with letters indicate the gluings.}
\label{fig:fundamental_domains}
\end{figure}
The Katok-Zemliakov construction (or unfolding procedure) of the billiard in the fundamental domain gives a surface $\Xsurf(a,b)$ made of $4$ reflected copies of the fundamental domain (see Figure~\ref{fig:unfolding_the_billiard}). The surface $\Xsurf(a,b)$ was studied in the particular case $a=b=1/2$ by different authors \cite{LS}, \cite{S}, \cite{FMZ}, \cite{EKZ1} and is called in this particular case the $6$-escalator (see Figure~\ref{fig:unfolding_the_billiard_L} for the origin of the name).
\begin{lemma} \label{lem:X_covers_L}
\label{lemma:X_is_normal_over_L}
The surface $\Xsurf(a,b)$ is a genus $5$ surface in $\HHH(2^4)$. It is a normal unramified cover of the surface $\Lsurf(a,b)$ with a Deck group $K$ isomorphic to the Klein four-group $K=\Z/2 \times \Z/2$.
\end{lemma}

\begin{proof}
The billiard table $\Tsurf(a,b)$ is invariant under horizontal and vertical reflections as well as the billiard in a fundamental domain. It is then straightforward to show that $\Xsurf(a,b)$ is an unramified normal cover of $\Lsurf(a,b)$ with group $\Z/2 \times \Z/2$. A direct computation shows that $\Xsurf(a,b)$ has $4$ singularities of angle $6 \pi$ (see Figure~\ref{fig:unfolding_the_billiard}).
\end{proof}

\begin{figure}[!ht]
\centering
\subfloat[Unfolding the toric fundamental domain of Figure~\ref{fig:fundamental_domain_torus}.]{%
\label{fig:unfolding_the_billiard_torus}%
\picinput{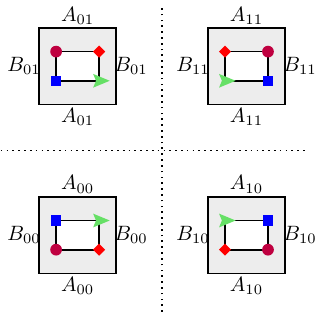}%
}
\hspace{1.3cm}
\subfloat[Unfolding the L shaped fundamental domain of Figure~\ref{fig:fundamental_domain_L}.]{%
\label{fig:unfolding_the_billiard_L}%
\picinput{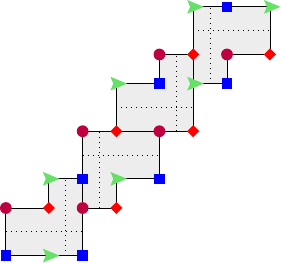}
}
\caption{Two versions of the surface $\Xsurf(a,b)$ obtained by unfolding the billiard in a fundamental domain. The gluings of edges are indicated by labels in case of ambiguity.}
\label{fig:unfolding_the_billiard}
\end{figure}

\subsection{The surface $\Xsurf_\infty(a,b)$ as a $\Z^2$ cover of $\Xsurf(a,b)$}
\label{section:Xinfty_as_a_Z2_cover}
As we did for unfolding the fundamental domain of the infinite billiard, we consider the unfolding of the whole billiard table $\Tsurf(a,b)$. The unfolding leads to a non compact surface that we denote $\Xsurf_\infty(a,b)$ which is made of four copies of the initial billiard. As the unfolding commutes with the action of $\Z^2$ the surface $\Xsurf(a,b)$ is also the $\Z^2$ quotient of $\Xsurf_\infty(a,b)$. We use this description to rewrite the distance in Theorem~\ref{thm:divergence_rate_billiard} as a pairing between a geodesic segment in $\Xsurf(a,b)$ with a cocycle in $H^1(\Xsurf;\Z^2)$.

\begin{wrapfigure}{r}{6cm}
\centering
\picinput{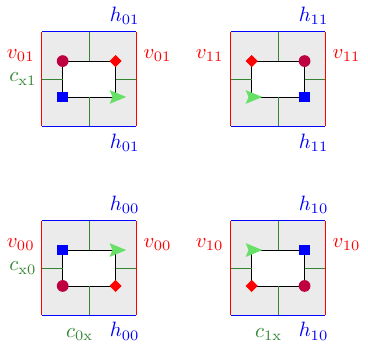}
\caption{Homology generators for $\Xsurf(a,b)$.}
\label{fig:homology_generators_of_X}
\end{wrapfigure}
We first build a system of generators for the homology of $\Xsurf(a,b)$. We label each copy of the torus fundamental domain in $\Xsurf(a,b)$ by $00$, $01$, $10$ and $11$ (see Figure~\ref{fig:unfolding_the_billiard_torus} and \ref{fig:homology_generators_of_X}). For $\kappa \in \{00,10,01,11\}$ let $h_\kappa$ (resp. $v_\kappa$) be the horizontal (resp. vertical) simple closed curve that delimit each copy (the exterior boundary). The curves $h_\kappa$ (resp. $v_\kappa$) have holonomy $1$ (resp. $i$). The automorphism group $K \simeq \Z/2 \times \Z/2$ of $\Xsurf(a,b)$ acts on the indices of $h_\kappa$ and $v_\kappa$ by addition (where we consider $0$ and $1$ as elements of $\Z/2$). The intersection form $\langle .,. \rangle$ on $\Xsurf(a,b)$ is such that $\langle h_\kappa, v_{\kappa'} \rangle = \delta_{\kappa,\kappa'}$ where $\delta_{ij}$ is the Kronecker symbol. In other words, the module generated by the elements $h_\kappa$ and $v_\kappa$ is a symplectic submodule and $\{(h_\kappa,v_\kappa)\}_\kappa$ is a symplectic basis. Moreover, this $\Z$-submodule is invariant under the action of $K$ (but not irreducible, see Lemma~\ref{lem:homology_splitting} below).

We consider four more elements of $H_1(\Xsurf; \Z)$. Let $c_{\mathrm{x}0}$ (resp. $c_{\mathrm{x}1}$)  be the circumferences of the horizontal cylinder that intersects the two copies $00$ and $10$ (resp. $01$ and $11$) of the torus fundamental domain. The curves $c_{\mathrm{x}0}$ and $c_{\mathrm{x}1}$ have both holonomy $(2-2a,0)$. We define as well the curves $c_{0\mathrm{x}}$ and $c_{1\mathrm{x}}$ with respect to the vertical cylinders. The curves $c_{0\mathrm{x}}$ and $c_{1\mathrm{x}}$ have both holonomy $(0,2-2b)$. As before the action of $K$ as automorphism group of $\Xsurf(a,b)$ corresponds to an action on indices of $c_{ij}$ if we set $0 \cdot \mathrm{x} = 1 \cdot \mathrm{x}=\mathrm{x}$.

There are two relations in $H_1(\Xsurf;\Z)$ among the curves defined above.
\begin{equation}
\begin{array}{l}
c_{\mathrm{x}0} - c_{\mathrm{x}1} = h_{00} - h_{01} + h_{10} - h_{11} \\
c_{0\mathrm{x}} - c_{1\mathrm{x}} = v_{00} - v_{10} + v_{01} - v_{11}
\label{eq:relation_for_homology_basis}
\end{array}.
\end{equation}
We have the following elementary
\begin{lemma}
The relations (\ref{eq:relation_for_homology_basis}) are the only relations in the family $\{h_{ij},v_{ij},c_{\mathrm{x}j},c_{i\mathrm{x}}\}$. Let $E_h$ (resp. $E_v$) be the span of $\{h_{ij},c_{\mathrm{x}k}\}_{i,j,k \in \{0,1\}}$ in $H_1(\Xsurf(a,b);\Z)$ (res. of $\{v_{ij},c_{k\mathrm{x}}\}_{i,j,k \in \{0,1\}}$ in $H_1(\Xsurf(a,b);\Z)$) then $H_1(\Xsurf(a,b);\Z) = E_h \oplus E_v$ and the sum is orthogonal with respect to the intersection form.
\end{lemma}

The infinite cover $\Xsurf_\infty(a,b) \rightarrow \Xsurf(a,b)$ corresponds to a certain subgroup $H$ of $\pi_1(\Xsurf(a,b))$ such that $\pi_1(\Xsurf(a,b)) / H \simeq \Z^2$. But as the cover is normal and $\Deck(\Xsurf_\infty(a,b)/\Xsurf(a,b)) \simeq \Z^2$ is an Abelian group, there exists a factorisation through the Abelianisation $H_1(\Xsurf(a,b); \Z)$ of $\pi_1(\Xsurf(a,b))$  (see also \cite{HoWe} for the description of $\Z$-cover). In other terms the cover is defined by an element of $H^1(\Xsurf(a,b);\Z^2)$ and more precisely we have the following explicit description.
\begin{lemma} \label{lem:vector_of_the_infinite_cover}
The $\Z^2$ covering $\Tsurf(a,b)/\Xsurf(a,b)$ is given by the dual $f \in H^1(X; \Z^2)$ with respect to the intersection form of the cycle
\[
\begin{pmatrix}
v_{00} - v_{10} + v_{01} - v_{11} \\
h_{00} - h_{01} + h_{10} - h_{11}
\end{pmatrix}
\in H_1(X; \Z^2).
\]
In other words, the subgroup of $\pi_1(\Xsurf(a,b))$ associated to the covering is the kernel of
\[
 \pi_1(\Xsurf(a,b)) \xrightarrow[]{Ab}
 H_1(\Xsurf(a,b);\Z) \xrightarrow[]{f}
 \Z^2.
\]
\end{lemma}

\begin{proof}
As before we consider the surface decomposed into four copies of the torus fundamental domain labeled $00$, $10$, $01$ and $11$. The labeling fits the action of the Klein $4$ group on the surface. Let $\gamma$ be a smooth curve in $\Tsurf(a,b)$ which follows the law of reflection when it hits an obstacle. Let $\overline{\gamma}$ its image in $\Xsurf(a,b)$. There is an ambiguity for the starting point of $\overline{\gamma}$ and we assume that we start in the copy $00$. Each time the curve $\overline{\gamma}$ hit a side associated to a vertical (resp. horizontal) scatter the curve $\overline{\gamma}$ switches from the copy $\kappa$ to $(1,0) \cdot \kappa$ (resp. $(0,1) \cdot \kappa$). At the same time, in the infinite table $\Tsurf(a,b)$ the curve $\gamma$ is reflected vertically (resp. horizontally). When the curve crosses a vertical (resp. an horizontal) boundary of the fundamental domain (labelled A (resp. B) in Figure~\ref{fig:fundamental_domain_torus}) the curve $\overline{\gamma}$ remains in the same copy. In other words, the endpoint of $\gamma$ in $T(a,b)$ only depends on the monodromy of $\overline{\gamma}$ with respect to $\Xsurf_\infty/\Xsurf$ and we need to consider only the case of the curves $\gamma=h_{ij},v_{ij}$ for $i=0,1$ and $j=0,1$.

As the copies $00$ and $01$ in $\Xsurf(a,b)$ corresponds to the absence of vertical reflection, the monodromy of $v_{00}$ and $v_{01}$ is $(1,0)$. Whereas for the copies $10$ and $11$, the curve $\gamma$ has been reflected and the monodromy of $h_{10}$ and $h_{11}$ is $(-1,0)$. The same analysis can be made for the curves $v_{ij}$ and the lemma follows from duality between $\{h_\kappa\}$ and $\{v_\kappa\}$.
\end{proof}
Now, we use the description of $\Xsurf_\infty(a,b) \rightarrow \Xsurf(a,b)$ in terms of homology to approximate the distance $d(p,\phi_t^\theta(p))$ of Theorem~\ref{thm:divergence_rate_billiard} in terms of intersection. But first of all, we need to approximative geodesic segment by elements of $H_1(\Xsurf,\Z)$.

For each triple $(p,\theta,t) \in \Xsurf \times S^1 \times \R_+$ we define an element $\gamma^\theta_t(p) \in H_1(\Xsurf;\Z)$ as follows. Consider the geodesic segment of length $t$ from $p$ in the direction $\theta$ and close it by a small piece of curve that does not intersect any curves $h_\kappa$ nor $v_\kappa$. The curve used to close the geodesic segment can be chosen to be uniformly bounded.

The proposition below shows that the distance of the particle in the billiard $\Tsurf(a,b)$ can be reduced to the study of the pairing of the approximative geodesic $\gamma^\theta_t(p)$ in $\Xsurf(a,b)$ and the cocycle $f \in H^1(X; \Z^2)$ defined in Lemma~\ref{lem:vector_of_the_infinite_cover}.
\begin{proposition}
Let $\|.\|_2$ be the Euclidean norm on $\R^2$. Let $p \in \Xsurf(a,b)$ and $\tilde{p} \in T(a,b)$ the lift of $p$ which belongs to the translate of the fundamental domain that contains the point $(0,0)$. Let $f \in H^1(X;\Z^2)$ be as in the previous lemma. Then
\[
\left\|\left\langle f, \gamma^\theta_t(p) \right\rangle - \phi_t^\theta(\tilde{p})\right\|_2 \leq \sqrt{2}.
\]
In particular
\[
\left| \| \langle f, \gamma^\theta_t(p) \rangle \|_2 - d(\tilde{p},\phi^\theta_t(\tilde{p}))\right| \leq \sqrt{2}.
\]
\end{proposition}

\begin{proof}
The distance between the point $\phi^\theta_t(p) \in \R^2$ and the associated level $\langle f, \gamma^\theta_t(p) \rangle \in \Z^2$ is bounded from above by the diameter of the fundamental domain. The latter is uniformly bounded by $\sqrt{2}$ (with respect to the parameters $a$ and $b$).
\end{proof}

As a consequence of the above proposition we reformulate our main result (Theorem~\ref{thm:divergence_rate_billiard}).
\begin{theorem} \label{thm:divergence_rate_surface}
Let $0 < a < 1$, $0 < b < 1$. For almost every $\theta$ and every $p \in \Xsurf(a,b)$, the approximative geodesic $\gamma^\theta_T(p)$ starting from $p$ in direction $\theta$ satisfies
\[
\limsup_{T \to \infty} \frac{\log \left| \left\langle f, \gamma^\theta_T(p) \right\rangle \right|}{\log T} = \frac{2}{3}.
\]
\end{theorem}

\subsection{Deck group action on $\Xsurf(a,b)$}
\label{subsection:deck_group_action}
We study the covering $\Xsurf(a,b)/\Lsurf(a,b)$ which is normal with Deck group the Klein four group $K = \Z/2 \times \Z/2$ by Lemma~\ref{lemma:X_is_normal_over_L}.

Let $v_{ij}$, $h_{ij}$, $c_{\mathrm{x}j}$ and $c_{i\mathrm{x}}$ for $i,j \in \{0,1\}$ be the generators of $H_1(\Xsurf;\Z)$ defined in Section~\ref{section:Xinfty_as_a_Z2_cover}. We identify them to vectors in the cohomology $H^1(\Xsurf;\Z)$ by duality. The action of the Klein four group $K$ on $\Xsurf(a,b)$ splits the homology in four subspaces. For the generators $\tau_v = (1,0)$ and $\tau_h = (0,1)$ of $K$ we define the subspace $E^{+-}$ to be the set of vectors $v \in H_1(\Xsurf;\Z)$ such that $\tau_v(v) = +1$ and $\tau_v(v) = -1$. We define similarly $E^{++}$, $E^{-+}$ and $E^{--}$.

We denote $h_K = h_{00} + h_{01} + h_{10} + h_{11}$ and $v_K = v_{00} + v_{01} + v_{10} + v_{11}$.
\begin{lemma} \label{lem:homology_splitting}
The action of the deck group of $\Xsurf(a,b) \rightarrow \Lsurf(a,b)$ splits the cohomology into four subspaces
\[
H^1(\Xsurf(a,b);\Q) = E^{++} \oplus E^{+-} \oplus E^{-+} \oplus E^{++},
\]
where each subspace $E^{\kappa}$ is defined over $\Q$ as follows
\begin{itemize}
\item $E^{++} = \Q\ [h_K] \oplus \Q\ [c_{\mathrm{x}0} + c_{\mathrm{x}1}] \oplus \Q\ [v_K] \oplus \Q\ [c_{0\mathrm{x}} + c_{1\mathrm{x}}] \simeq H_1(\Lsurf(a,b);\Q)$
\item $E^{+-} = \Q\ [h_{00} - h_{01} + h_{10} - h_{11}] \oplus \Q\ [v_{00} - v_{01} + v_{10} - v_{11}]$
\item $E^{-+} = \Q\ [h_{00} + h_{01} - h_{10} - h_{11}] \oplus \Q\ [v_{00} + v_{01} - v_{10} - v_{11}]$
\item $E^{--} = \Q\ [h_{00} - h_{01} - h_{10} + h_{11}] \oplus \Q\ [v_{00} - v_{01} - v_{10} + v_{11}]$
\end{itemize}
\end{lemma}
We emphasise that the invariant part of $H^1(\Xsurf(a,b);\Z)$ under the subgroup $\langle \tau_v \rangle \subset K$ is identified with $H^1(\Xsurf(a,b)/\langle \tau_v \rangle;\Z)$. This is the main reason for which we consider each quotient of $\Xsurf(a,b)$ by the three subgroups of order two generated by $\tau_v$, $\tau_h$ and $\tau_h\ \tau_v$.

\begin{lemma} \label{lem:cc_of_quotients_of_X}
The surfaces $\Xsurf(a,b) / \langle \tau_v \rangle$ and $\Xsurf(a,b) / \langle \tau_h \rangle$ belongs to the hyperelliptic component $\HHH^{hyp}(2,2)$ while the surface $\Xsurf(a,b) / \langle \tau_v \tau_h \rangle$ belongs to the hyperelliptic locus $\LLL \subset \HHH^{odd}(2,2)$.
\end{lemma}

\begin{proof}
We see on the two figures below that the central symmetry in each polygonal representation of the surfaces $\Xsurf(a,b) / \langle \tau_v \rangle$ and $\Xsurf(a,b) / \langle \tau_v \tau_h \rangle$ gives rise to an involution that does not preserve the directions: the direction $\theta$ is sent to $-\theta$. The quotient by such involution gives rise to quadratic differentials.
\begin{figure}[!ht]
\centering
\subfloat[Quotient of $\Xsurf(a,b)$ by $\tau_h$.]{%
\label{fig:quotient_X_h}%
\picinput{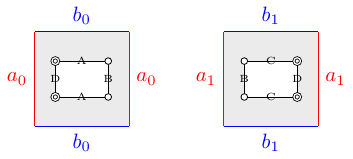}%
}
\hspace{1cm}
\subfloat[Quotient of $\Xsurf(a,b)$ by $\tau_h \tau_v$.]{%
\label{fig:quotient_X_hv}%
\picinput{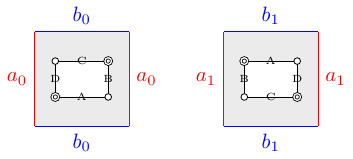}%
}
\caption{The quotients of degree $2$ of $\Xsurf(a,b)$.}
\label{fig:quotients_of_X}
\end{figure}
In both cases the quotient is a sphere endowed with a quadratic differential and hence the surfaces $\Xsurf(a,b) / \langle \tau_v \rangle$ and $\Xsurf(a,b) / \langle \tau_v \tau_h \rangle$ are hyperelliptic. In the first quotient, the two singularities are exchanged and hence $\Xsurf(a,b) / \langle \tau_v \rangle$ belongs to $\HHH^{hyp}(2,2)$ which corresponds to the orientation cover of quadratic differentials in $\QQQ(4,-1^8)$ (this is the definition of $\HHH^{hyp}(2,2)$, see~\cite{KZ}). While for $\Xsurf(a,b) / \langle \tau_v \tau_h \rangle$ the zeros are fixed by the involution and the surface belongs to the hyperelliptic locus $\LLL \subset \HHH^{odd}(2,2)$ which corresponds to the orientation cover of quadratic differentials in $\QQQ(1^2,-1^6)$.
\end{proof}

\section{Moduli space and Lyapunov exponents} \label{sec:computation_of_Lyapunov_exponents}
In this section, using McMullen classification of $\SL(2,\R)$-invariant closed set and probability measures in $\HHH(2)$, we classify the possible closures $\overline{\SL(2,\R) \cdot \Xsurf(a,b)}$ of $\SL(2,\R)$-orbits of the surfaces $\Xsurf(a,b)$ in $\HHH(2^4)$. Each closure carries a unique $\SL(2,\R)$-invariant ergodic probability measure and we compute the Lyapunov exponents of the Kontsevich-Zorich cocycle with respect to it.

\subsection{Moduli space and $\Xsurf(a,b)$} \label{section:moduli_for_X}
We recall that $\Xsurf(a,b) \in \HHH(2^4)$ is a cover of $\Lsurf(a,b) \in \HHH(2)$ (Lemma~\ref{lem:X_covers_L}). This property is preserved by the action of $\SL(2,\R)$ and more precisely the action of $\SL(2,\R)$ is equivariant: for any $g \in \SL(2,\R)$ the surface $g \cdot \Xsurf(a,b)$ is a cover of $g \cdot \Lsurf(a,b)$. Hence, all $\SL(2,\R)$-orbits of $\Xsurf(a,b)$ belongs to the sublocus of $\HHH(2^4)$ which corresponds to particular covering of surfaces in $\HHH(2)$. This locus, which we denote by $\Xablocus$, is a closed $\SL(2,\R)$-invariant subvariety of $\HHH(2^4)$ which is a finite cover of $\HHH(2)$. In particular, McMullen's classification Theorem for $\SL(2,\R)$-invariant closed subset and probability measures (Theorem~\ref{thm:McMullen_classification}) holds for closures of $\SL(2,\R)$-orbits of $\Xsurf(a,b)$.

The action of the Klein four-group $K$ on surfaces $\Xsurf(a,b)$ and the splitting of Lemma~\ref{lem:homology_splitting} holds for any surface $Y$ in $\Xablocus$. For any $Y \in \Xablocus$ we denote by $\overline{Y} = Y/K$ its quotient in $\HHH(2)$. We have maps $H_1(Y;\R) \rightarrow H_1(\overline{Y};\R)$ (resp. $H^1(\overline{Y};\R) \rightarrow H^1(Y;\R)$) which are equivariant with respect to the Kontsevich-Zorich cocycle. In particular the explicit decomposition in the first part of Lemma~\ref{lem:homology_splitting} remains valid for any surface $Y$ in $\Xablocus$ as it depends only of the action of $K$. In particular, we get an $\SL_2(\R)$-equivariant splitting of the Hodge bundle. But, as $H_1(\Xsurf(a,b);\Z)$ and $H_1(Y;\Z)$ can only be identified locally, the explicit basis of homology we have exhibited for $\Xsurf(a,b)$ has no meaning for $Y$.

\subsection{Computation of Lyapunov exponents} \label{section:computation}
In this section we compute the individual Lyapunov exponents of the KZ cocycle for all $\SL(2,\R)$-invariant ergodic measures on $\Xablocus$. We denote by $E \rightarrow \Xablocus$ the Hodge bundle over $\HHH(2^4)$ restricted to $\Xablocus$.

Recall, that the KZ cocycle is symplectic. Hence, the Lyapunov exponents come by pair of opposites $(\nu, -\nu)$. In the following we call \emph{non negative spectrum} of the KZ cocycle the non-negative numbers $1=\nu_1 > \nu_2 \geq \ldots \geq \nu_g$ such that the multiset $(\nu_1,\nu_2, \ldots ,\nu_g, -\nu_g,\ldots,-\nu_1)$ are the Lyapunov exponents of the KZ cocycle. In our case, for any surface $Y$ in $\Xablocus$ the Oseledets decomposition of $H_1(Y;\R)$ respect the splitting $E = E^{++} \oplus E^{+-} \oplus E^{-+} \oplus E^{--}$. Moreover, the maximal Lyapunov exponent of the KZ cocycle, which equals $1$, belongs to $E^{++}$. Hence the non negative spectrum can be written $\{1,\nu^{++},\nu^{+-},\nu^{-+},\nu^{--}\}$ where $\{1,\nu^{++}\}$ (resp. $\{\nu^{+-}\}$, $\{\nu^{-+}\}$ and $\{\nu^{--}\}$) is the non negative Lyapunov spectrum of the KZ cocycle restricted to $E^{++}$ (resp. $E^{+-}$, $E^{-+}$ and $E^{--}$).

\begin{theorem} \label{thm:value_of_lyap_exp}
For any $\SL(2,\R)$-invariant ergodic probability measure on $\Xablocus$:
\[
\nu^{++} = \nu^{--} = 1/3 \qquad \text{and} \qquad \nu^{+-} = \nu^{-+} = 2/3.
\]
\end{theorem}

\begin{proof}
We first consider the case of the rank $4$ subbundle $E^{++}$ which corresponds to invariant vectors under the action of the Klein four group $K$. $E^{++}$ identifies with the pullback of the Hodge bundle over $\HHH(2)$ and in particular, we deduce from results of M.~Bainbridge~\cite{Ba} (see also Theorem~\ref{thm:exp_lyap_hyp}) that $\nu^{++} = 1/3$.

We now consider the case of the rank $2$ subbundles $E^\kappa$ for $\kappa \in \{--,+-,-+\}$. Lemma~\ref{lem:cc_of_quotients_of_X} implies that the subbundle $E^{++} \oplus E^{--}$ (resp. $E^{++} \oplus E^{+-}$ and $E^{++} \oplus E^{-+}$) can be identified to pullback of bundles over different covering loci in $\HHH^{hyp}(2,2)$ and $\LLL \subset H^{odd}(2,2)$. More precisely, the quotient map $Y \mapsto Y / \langle \tau_h \tau_v \rangle$ (resp. $Y \mapsto Y / \langle \tau_v \rangle$ and $Y \mapsto Y / \langle \tau_h \rangle$) induce an isomorphism between $E^{++} \oplus E^{--}$ (resp $E^{++} \oplus E^{+-}$ and $E^{++} \oplus E^{-+}$) and respectively $H_1(Y / \langle \tau_h \tau_v \rangle;\Z)$ (resp. $H_1(Y / \langle \tau_h \rangle;\Z)$ and $H_1(Y / \langle \tau_v \rangle;\Z)$). 

To compute the remaining Lyapunov exponents, we use twice Theorem~\ref{thm:exp_lyap_hyp}. The loci $\HHH^{hyp}(2,2)$ and $\LLL$ comes from orientation coverings of surfaces in the quadratic strata respectively $\QQQ(4,-1^8)$ and $\QQQ(1^2,-1^6)$. For those two components, using Theorem~\ref{thm:exp_lyap_hyp}, we get that the sum of positive Lyapunov exponents are respectively
\begin{align*}
\nu_1 + \nu_2 + \nu_3 &= \frac{2/3+6}{4} = 5/3 & \text{for $\LLL \subset \HHH^{odd}(2,2)$} \\
\nu_1 + \nu_2 + \nu_3 &= \frac{8}{4} = 2 & \text{for $\HHH^{hyp}(2,2)$}.
\end{align*}
By subtracting $4/3=1+1/3$ to each of them that corresponds to the contribution of $E^{++}$ we get that that $\nu^{--} = 1/3$ and $\nu^{+-} = \nu^{-+} = 2/3$.

Notice that the assumption of \emph{regularity} for the measures in Theorem~\ref{thm:exp_lyap_hyp} are automatically satisfied. The regularity property is stable under taking covering locus and is satisfied for Haar measures on Veech surfaces and Lebesgue measure on a stratum (see~\cite{EKZ2}).
\end{proof}

Applying the main result of Chaika-Eskin~\cite{CE} one obtain directly from Theorem~\ref{thm:value_of_lyap_exp} the following.
\begin{corollary}[\cite{CE}] \label{cor:strong_Oseledets}
For all surfaces $Y$ in $\Xablocus$, for almost every $\theta in [0,2\pi)$, the surface $e^{i \theta} Y$ is Oseledets generic: the Lyapunov exponents exist and their values are described by Theorem~\ref{thm:value_of_lyap_exp}.
\end{corollary}

\subsection{A remark on the equality $\nu^{+-} = \nu^{-+}$}
In Theorem~\ref{thm:value_of_lyap_exp}, we have equality of Lyapunov spectrum in $E^{+-}$ and $E^{-+}$. In that section we deduce the equality from the symmetries of the wind-tree model.

An \emph{automorphism} (respectively \emph{anti-automorphism}) of a translation surface $(S,\omega)$ is an affine diffeomorphism $\phi: S \rightarrow S$ such that $\phi^* \omega = \omega$ (resp. $\phi^* \omega = - \omega$). Quotient by automorphisms or anti-automorphisms correspond to cover of translation surfaces (with either an Abelian or quadratic differential). Automorphism and anti-automorphisms commutes with the $\SL(2,\R)$ actions (they are precisely the affine diffeomorphisms for which the derivative belongs to the center of $\SL(2,\R)$ which is $\pm id$) and this was used in the preceding sections to build the decomposition $E = E^{++} \oplus E^{+-} \oplus E^{-+} \oplus E^{--}$ (see Lemma~\ref{lem:homology_splitting}).

An \emph{elliptic symmetry} of a translation surface $(S,\omega)$ is a diffeomorphism $\phi: S \rightarrow S$ such that $\phi^* \omega = e^{i \theta} \omega$ for some $\theta \in [0,2\pi) \backslash \{0,\pi/2,\pi,3\pi/2\}$. Such map $\phi$ necessarily preserves the underlying complex structure of $\omega$ and implies that $\phi$ is of finite order. Identically, a translation surface has an elliptic symmetry if it comes from an orientation cover of higher order differentials. Elliptic symmetries \emph{do not} commute with the $\SL(2,\R)$-action but only with $\SO(2)$.

An \emph{elliptic locus} is an orientation cover of a connected component of differential of order larger than $>2$. Equivalently, it is a subset of a connected component of stratum with an elliptic symmetry which is locally constant.

In $\HHH(2^4)$, we consider the locus of surfaces which are cover of L-shaped surfaces with a rotation symmetry of order $4$. Examples of such surfaces are the surfaces $\Xsurf(a,a)$. The quotient of an element of that locus is a quartic differential on the sphere with angles $\pi$, $\pi$, $-3/4\pi$ and $-1/4\pi$. The real dimension of that locus is $4$ (or $3$ in $\HHH^{(1)}(2^4)$). The following lemma may be used to conclude to equality of Lyapunov exponents in $E^{+-}$ and $E^{-+}$.

\begin{lemma}
Let $X \subset \HHH_g(\alpha)$ be a closed $\SL(2,\R)$-invariant affine locus in $\HHH_g(\alpha)$ of real dimension $2d$ that supports an ergodic measure $\mu$. Let $Y \subset X$ be an elliptic locus of real dimension $d$ and denote by $G$ the associated action on $H^1(S,\R)$ over an element $(S,\omega)$ in $Y$. Let $E_i$ and $E_j$ be two subbundles of the Hodge bundle that are $\SL(2,\R)$ invariant and such that for some $g \in G$ we have $g E_i = E_j$ in the intersection of a neighborhood of $(S,\omega)$ and $Y$. Then the spectrum of the Kontsevich-Zorich cocycle restricted to $E_i$ and $E_j$ coincide.
\end{lemma}

\begin{proof}
We prove that in the intersection of the neighborhood of $(S,\omega)$ with $Y$ there is a generic surface for $\mu$.

Let $\phi$ be an elliptic symmetry of $(S,\omega)$. In other words $\phi^* (\omega) = e^{i \theta} \omega = \cos(\theta) \Re(\omega) - \sin(\theta) \Im(\omega) + i (\sin(\theta) \Re(\Omega) + \cos(\theta) \Im(\omega))$. It is clear that if $\sin(\theta) \not= 0$, then this equation is not preserved along the stable variety of $\omega$ (ie along a deformation of $\Im(\omega)$ only) and along the geodesic flow. Hence, at each point $(S,\omega)$ of $Y$ the subspace $Y$, the stable variety of $(S,\omega)$ and its geodesic are transverse. From the dimension condition, almost all geodesic encounter the stable variety of some element in $Y$. In particular, almost all points in $Y$ (for its $d-1$-dimensional Lebesgue measure) have generic Teichm\"uller geodesics.

To conclude to equality, let us consider a generic surface in $Y$. There exists a generic direction $\theta'$ in $G$ such that $\theta + \theta'$ is also generic ($\theta$ is such that $\phi^* \omega = e^{i \theta} \omega$). Because $\phi^* E_i = \phi^* E_j$ at $(S,\omega)$ we do have equality of exponents in $E_i$ and $E_j$.
\end{proof}

\section{Deviations for translation surfaces} \label{section:deviations}
In this section, we prove Theorem~\ref{thm:deviations} which concerns growth of geodesics.

We recall notation from the introduction. Let $\HHH(\alpha)$ be a stratum of Abelian differential and $\mu$ a $g_t$-invariant ergodic measure on $\HHH(\alpha)$. We denote by $1 = \nu_1 > \nu_2 > \ldots > \nu_k$ the positive Lyapunov exponents of the KZ cocycle and for $X \in \HHH(\alpha)$ which is Oseledets generic
\begin{equation} \label{eq:Oseledets_decomposition}
H^1(X;\R) = F^u_1 \supset F^u_2 \supset \ldots \supset F^u_k \supset F^u_{k+1} = F^c \supset F^s_k \supset \ldots \supset F^s_1 \supset F^s_0 = \{0\}
\end{equation}
the associated Oseledets flag. By Oseledets theorem, the decomposition (\ref{eq:Oseledets_decomposition}) is measurable and is invariant under the Teichm\"uller flow.

We want to prove the following statement: for $\mu$-almost all $X$ which are Oseledets generic, for all $p \in X$ with infinite forward orbit and any norm on $H^1(X;\R)$
\begin{enumerate}
\item for all $f \in F^u_i \backslash F^u_{i+1}$
\[
\limsup_{T \to \infty} \frac{\log | \langle f, \gamma_T(p)) \rangle|}{\log T} = \nu_i,
\]
\item for $f \in F^c$
\[
\limsup_{T \to \infty} \frac{\log |\langle f, \gamma_T(p) \rangle |}{\log T} = 0,
\]
\item there exists a constant $C$, such that for $f \in F^s_i \backslash F^s_{i-1}$
\[
\forall T \geq 0,\ | \langle f, \gamma_T(p) \rangle | \leq C.
\]
\end{enumerate}

We first notice that to prove Theorem~\ref{thm:deviations}, by ergodicity of the Teichm\"uller flow, it is enough to prove it for surfaces $X$ belonging in a small open set of $\HHH(\alpha)$ of positive measure. The strategy is as follows. We build a small open set in which we have uniform estimates for the linear flows. Next, for a surface in this small open set we consider long pieces of trajectory under the linear flow that we decompose using the KZ cocycle. Then, using the uniform estimates, we get the lower and upper bounds.

\subsection{Transversals for the Teichm\"uller flow} \label{subsection:transversal}
In order to code geodesics in an individual surface we use Veech's construction of zippered rectangles~\cite{Ve1}. We recall our convention of markings (see Section~\ref{section:background}) that each translation structure in $\HHH(\alpha)$ carry a choice of an horizontal separatrix at each singularity.

A surface in $\HHH(\alpha)$ is called \emph{regular} if there is no saddle connection in both horizontal and vertical directions. In a regular surface the linear flow in vertical direction is minimal (Keane's Theorem~\cite{Ke}). If there is a connection in vertical (resp. horizontal) direction then the forward (resp. backward) orbit for the Teichm\"uller flow goes to infinity. In particular, using Poincar\'e recurrence theorem, we get that the set of regular surfaces is a set of full measure for $\mu$.

Let $X$ be a regular surface in the support and $\Sigma \subset X$ the finite set of singularities of $X$. Following \cite{Ve1}, we decompose the surface into zippered rectangles. Recall that there is a marked outgoing separatrix in $X$. We consider the initial segment of length $1$ on this separatrix that we identify with $[0,1]$. The Poincar\'e map of the linear flow in this segment is an interval exchange transformation. There exists a canonical segment $I \subset X$ built from Rauzy induction (\cite[Proposition~9.1]{Ve1}, see also \cite{Vi}). The rectangles above each domain of continuity of the interval exchange transformation on $I$ give a decomposition $X = \bigcup R_j$ where $R_j$ are geodesic rectangles with horizontal sides inside $I$ and vertical sides which contain singularities or hit a singularity in the future. The number of rectangles is $d = 2g-2+s-1$ where $g$ is the genus of $X$ and $s$ the number of singularities.

Let $X$ be a regular surface in the support of $\mu$ and $X = \bigcup R_j$ its decomposition into zippered rectangle. The parameters of the zippered rectangles (lengths and heights of the rectangles) give local coordinates for $\HHH(\alpha)$ in a neighbourhood of $X$. In particular, we get zippered rectangles construction for surfaces near $X$ that are \emph{not} regular. Let $U \subset \mathcal{H}(\kappa)$ be an open set which contains $X$ and for which the zippered rectangles obtained from $X$ gives a chart of $\HHH(\alpha)$. In $U$, we have a trivialisation of the Hodge bundle and we identify all fibers with $H^1(X;\R)$.

To each rectangle $R_j$ of a zippered rectangle decomposition of a surface $Y$ in $U$ is associated a curve $\zeta_j \subset Y \backslash \Sigma$ (up to homotopy in $Y \backslash \Sigma$) which corresponds to the Poincar\'e map on the canonic interval of $Y$. The vertical holonomy of $\zeta_j$ is the height of $R_j$. The following is a classical fact.
\begin{lemma}
The set $\{\zeta_j\}_{j=1,\ldots,d}$ forms a basis of $H_1(Y \backslash \Sigma; \Z)$.
\end{lemma}
Let $Y \in U$ and $I \subset Y$ be the canonical transversal for the linear flow of $Y$. To a point $p$ in $I$, we associate the sequence of return times $T_n = T_n(p)$ of the linear flow into $I$. Each curve $\gamma_{T_n}(p)$ have both ends in $I$ and we close it using a small piece of the horizontal segment contained in $I$. For any $p \in I$ with infinite orbit and any $n$ we have a unique decomposition as concatenation of curves
\[
\gamma_{T_n}(p) = \zeta_{j_1}(p)\ \zeta_{j_2}(p)\ \zeta_{j_3}(p)\ \ldots\ \zeta_{j_n}(p)
\]
and hence, there exists some non negative numbers $m_{T_n,j}(p)$ for $j=1,\ldots,d$ such that, in homology
\[
\gamma_{T_n}(p) = \sum_{j=1}^d m_{T_n,j}(p) \zeta_j \in H_1(S \backslash \Sigma).
\]
Let $p \in Y$ with infinite backward and forward orbit. There is a unique point $p' \in I$ such that the orbit of $p'$ under the linear flow goes to $p$ before returning in $I$. For $T \geq 0$, we denote by $\gamma_T(p)$ the curve $\gamma_{T_n}(p')$ where $T_{n-1}(p') < T \leq T_n(p')$.

\begin{lemma} \label{lem:transversal}
We can choose $U$ in such way that there exist constants $K_1$ and $K_2$ such that for all $Y \in U$
\begin{enumerate}
\item for $j=1,\ldots,d$, the length $l_j$ and height $h_j$ of the rectangle $R_j = R_j(Y)$ satisfy $K_1^{-1} < h_j < K_1$ and $K_1^{-1} < l_j < K_1$ ,
\item for every point $p \in I$, the decomposition of the geodesic $\gamma_{K_2}(p) = \sum c_j \zeta_j$ in $H_1(Y \backslash \Sigma;\Z)$ where the sum is for $j$ between $1$ and $d$ is such that no $c_j$ is zero. In other words, any geodesic longer than $K_2$ goes through all rectangles $R_j$.
\end{enumerate}
\end{lemma}

\begin{proof}
We consider the interval exchange transformation on the segment $I$ associated to surfaces $Y$ in a neighborhoud of $X$. We recall from \cite{Ke}, that by doubling all points of $I$ which are preimages of a discontinuity of the interval exchange gives a Cantor set. The interval exchange transformation has a prolongation to this Cantor set and is a homeomorphism which is semi-conjugated to the initial transformation on $T$.

Around $X$, the lengths $\ell_j$ and heights $h_j$ are continous function of $Y$. Hence, to fulfill the first condition it is enough to consider a relatively compact open set $U$ inside the chart given by the zippered rectangles of $X$.

We prove that it is possible to satisfy the second one. Because of the regularity condition, the linear flow of $X$ is minimal (Keane's Theorem \cite{Ke}). Let $I \subset X$ be the segment associated to $X$. For any $p \in I$ with infinite future orbit, there exists a time $T = T(p)$ such that the curve $\gamma_T(p)$ has visited all rectangles. We choose $T(p)$ to be the first return time of $p$ in $I$ with this property. The map $p \mapsto T(p)$ is locally constant on $I$ and hence also on the associated Cantor set $\widetilde{I}$. By minimality, $T(p)$ is uniformly bounded, otherwise there would exists a point $p$ for which its future does not cross some rectangle. Hence, on $X$, any curve of length longer than $K = \max_{p \in I} T(p) < \infty$ goes through all rectangles. In a small neighborhood of $X$, the rectangles associated to the time $K$ are still rectangles and their heights have been modified continuously with respect to the surface. By choosing $U$ small enough we may ensure that all rectangles of length less than $K$ in $Y$ are still rectangles in $Y' \in U$ and their heights are uniformly bounded by $K_2 = K + \varepsilon$ with $\varepsilon > 0$.
\end{proof}

By taking smaller $U$ if necessary, we assume that it is ``flow box'' that contains $X$. Namely, $U$ is identified with a transversal $P$ to the Teichm\"uller flow containing $X$ times an interval $]-\epsilon; \epsilon[$. For $Y \in P$, we consider the return times $t_n = t_n(Y)$ of the surface $Y$ in $P$. By ergodicity of $\mu$, this is well defined for $\mu$-almost all $Y$ in $P$. The segment $I$ in $g_{t_n} Y$ becomes a segment of length $e^{-t_n} I$ in $Y$. We define the curves $\zeta^{(n)}_j$ as the first return time associated to $e^{-t_n} I$ in $Y$. In homology, we have $\zeta^{(n)}_j = \left(B^{(t_n)}\right)^*(\zeta_j)$.

\subsection{Upper bound}
Let $X$ be a regular surface and $P$ a transversal to the Teichm\"uller flow containing $X$ as in section~\ref{subsection:transversal}. In particular we assume that lemma~\ref{lem:transversal} holds for all surfaces $Y$ in $P$.

Let $Y \in P$ a surface which is recurrent for the Teichm\"uller flow and let $\zeta^{(n)}_j$, $j=1,\ldots,d$, $n=0,1,\ldots$ be the curves defined at the end of the previous section as first return times into $I^{(n)} = e^{-t_n} I \subset Y$. By linear algebra, we know that the parings $\left\langle f, \zeta_j^{(n)} \right\rangle = \left\langle B^{(t_n)} f, \zeta_j \right\rangle$ is bounded by $\|B^{(t_n)} f\|$ times a constant. To prove that the bound still holds for a geodesic not necessarily of the form $\zeta^{(n)}_j$ we use the following lemma.
\begin{lemma}[\cite{Fo} Lemma~9.4, \cite{Zo1} Proposition~8] \label{lem:cut_into_pieces}
Let $X$ be a regular surface and $P$ a transversal containing $X$ as above. Let $Y \in P$ be recurrent for the Teichm\"uller flow and $\zeta^{(n)}_j$ be the curves defined by first return times. Let $p \in I \subset Y$ be a point with infinite future orbit. For each $T \geq 0$ there exists an integer $n=n(T)$ and a decomposition
\[
\gamma_T(p) = \sum_{k = 0}^n \sum_{j = 1}^d m^{(k)}_j \zeta_j^{(k)} \qquad \text{in $H_1(Y;\Z)$,}
\]
which satisfies
\begin{enumerate}
\item the $m^{(k)}_j$ are non negative integers for $k=1,\ldots,n$ and $j=1,\ldots,d$,
\item $\sum_j m^{(n)}_j \not= 0$ and $K_1^{-1}\, e^{t_n} < T$ \label{DK:estim_on_n},
\item $K_1^{-1}\, e^{t_k} \leq \ell \left(\zeta^{(k)}_j \right) \leq K_1\, e^{t_k}$ for $j=1,\ldots,d$, where $\ell$ denotes the length,
\item $\sum_j m^{(k)}_j \leq 2 (K_1)^2 e^{t_{k+1}-t_k}$.
\end{enumerate}
\end{lemma}

\begin{proof}
Let $I=I^{(0)} \subset Y$ be the segment of the interval exchange transformation associated to the zippered rectangles decomposition of $Y$. For $k \geq 1$, let $I^{(k)}$ be the subintervals of $Y$ which are the image of $I(g_{t_k}Y) \subset g_{t_k}(Y)$ under $g_{-t_k}$. We recall that this segment is the segment with the same left extremity as $I$ but whose length is $e^{-t_k}$ times the one of $I$.

We describe the so called \emph{prefix-suffix} decomposition in symbolic dynamics for $\gamma_T(p)$. We assume that $T$ is a return time of the linear flow in $I$ and note $\gamma = \gamma_T(p)$. Let $n$ be the largest $k$ such that the closed curve $\gamma$ crosses twice $I^{(k)}$. We may decompose $\gamma = r_-^{(n)}\ \gamma^{(n)}\ r_+^{(n)}$ where
\begin{itemize}
\item $r_-^{(n)}$ starts from $p$ and ends in $I^{(n)}$,
\item $\gamma^{(n)}$ is a non empty concatenation of $\gamma^{(n)}_j$,
\item $r_+^{(n)}$ starts from $I^{(n)}$ and ends at $p$.
\end{itemize}
We choose $r_-^{(n)}$ and $r_+^{(n)}$ to be minimal and hence their length are smaller than $K_1 e^{t_n}$ by Lemma~\ref{lem:transversal}, property 1. On the other hand, by definition, each curve $\gamma^{(n)}_j$ is of length larger than $K_1^{-1} e^{t_n}$ and hence $T > K_1^{-1} e^{t_n}$. This proves equation 2 and also 1 and 3 for the special case $k=n$.

We now proceed by induction and decompose $r_-^{(n)}$ and $r_+^{(n)}$ with respect to the other recurrence times $0 < t_k < t_n$. We assume that we built two curves $r_-^{(k)}$ and $r^{(k)}_+$ and two sequences $\gamma_-^{(k)},\gamma_-^{(k+1)},\ldots,\gamma_-^{(n-1)}$ and $\gamma_+^{(k)}, \gamma_+^{(k+1)}, \ldots, \gamma_+^{(n-1)}$ such that
\[
\gamma = r_-^{(k)} \ \left(\gamma_-^{(k)} \gamma_-^{(k+1)} \ldots \gamma_-^{(n-1)}\right)\  \gamma^{(n)}\ \left(\gamma_+^{(n-1)} \ldots \gamma_+^{(k+1)} \gamma_+^{(k)}\right)\ r_+^{(k)}
\]
with
\begin{itemize}
\item $\gamma_-^{(m)}$ starts from $I^{(m)}$, ends in $I^{(m+1)}$ and does not cross $I^{(m+1)}$ before its endpoint,
\item $\gamma_+^{(m)}$ starts from $I^{(m+1)}$, ends in $I^{(m)}$ and does not cross $I^{(m+1)}$ after its startpoint,
\item $r_-^{(k)}$ starts from $I^{(0)}$ and ends in $I^{(k)}$ and does not cross $I^{(k)}$ before its endpoint,
\item $r_+^{(k)}$ starts from $I^{(k)}$ and ends in $I^{(0)}$ and does not cross $I^{(k)}$ after its startpoint.
\end{itemize}
Repeating the same procedure, we end with the decomposition
\[
\gamma = \gamma_-^{(0)}\ \gamma_-^{(1)}\ \ldots\ \gamma_-^{(n-1)}\  \gamma^{(n)}\ \gamma_+^{(n-1)}\ \ldots\ \gamma_+^{(1)}\ \gamma_+^{(0)}.
\]
From the construction, we know that both $\gamma_-^{(k)}$ and $\gamma_+^{(k)}$ may be written as concatenation of curves $\zeta^{(k)}_j$ with positive coefficients. Let $\gamma_-^{(k)} + \gamma_+^{(k)}= \sum m^{(k)}_j \zeta_j^{(k)}$ in $H_1(Y;\Z)$. The property 1 is clearly satisfied. By maximality $\gamma_-^{(k)}$ satisfy $\ell\left(\gamma_-^{(k)} \right) < K_1 e^{t_{k+1}}$ and the same is true for $\gamma_+^{(k)}$. On the other hand each $\zeta^{(k)}_j$ is of length at least $K_1^{-1} e^{t_k}$ by property 1 of Lemma~\ref{lem:transversal} and hence property 3 is satisfied. From the latter inequalities, we obtain
\[
K_1^{-1} e^{t_k} \sum m^{(k)}_j < \sum m^{(k)}_j \ell(\zeta_j^{(k)}) < 2 K_1 e^{t_{k+1}}.
\]
which implies that $\sum m^{(k)}_j < 2 (K_1)^2 e^{t_{k+1}-t_k}$ which is property 4.
\end{proof}

Now, we prove the upper bound in Theorem~\ref{thm:deviations}. We restrict to the case 1 relative to one of the unstable subspace $F^u_i$ of the Oseledets flag. The same proof works for the other cases. We follow mainly Section~9 of \cite{Fo} (see also Section~6 of \cite{Zo1} and Section~4.9 of \cite{Zo2}). In what follows $K_i$ for $i=3,4,\ldots$ denote constants which do not depend on the time $T$ or the number $n = n(T)$.

We fix $Y \in P$ which is generically recurrent and Oseledets generic. We also fix $p$ a point in $Y$ with infinite forward orbit for the linear flow. By Lemma~\ref{lem:cut_into_pieces} and Oseledets' theorem (see section~\ref{section:introduction}), for any $\epsilon > 0$, there exists a constant $K_3$ and such that for $T$ big enough the following estimation holds
\begin{equation} \label{eq:estim1}
\left|\left\langle f, \gamma_T(p) \right\rangle\right| \leq K_3 \sum_{k=1}^n e^{t_{k+1}-t_k} \ e^{(\nu_i+\varepsilon) t_k}.
\end{equation}
Moreover, if $Y$ is generically recurrent for the Teichm\"uller flow, one can ensure that
\[
\lim_{k \to \infty} \frac{t_k}{k} = M
\]
where $M$ is the inverse of the transverse $\mu$-measure of $P$. As $Y$ is generically recurrent, for any $\delta$ the following estimation holds for $k$ big enough
\begin{equation} \label{eq:approx_mult_return_time}
\left(M-\delta \right) k \leq t_k \leq \left(M+ \delta \right) k
\end{equation}
Using~(\ref{eq:estim1}) and~(\ref{eq:approx_mult_return_time}) we get that for $T$ big enough we get
\begin{align*}
\left| \left\langle f, \gamma_T(p) \right \rangle \right|
 & \leq K_4 \sum_{k=1}^n \exp \left( (M+\delta)(k+1) - (M-\delta)k\right) \ 
                      \exp\left (\nu_i + \epsilon) (M+\delta) k \right) \\
 & \leq K_5 \sum_{k=1}^n \exp \left((\nu_i+\epsilon)(M+\delta)+2\delta)k \right) \\
 & \leq K_6 \exp (((\nu_i+\epsilon)(M+\delta)+2\delta)n).
\end{align*}
Now, by equation~(\ref{eq:approx_mult_return_time}) and the choice of $n=n(T)$ in the estimate \ref{DK:estim_on_n} in Lemma~\ref{lem:cut_into_pieces} we have for $T$ big enough
\[
\exp \left((M - \delta ) n \right) \leq \exp(t_n) \leq K_1\ T
\]
Hence we get that for $T$ big enough
\[
\frac{\log |\langle f, \gamma_T(p)|}{\log T} \leq \frac{(\nu_i+\epsilon)(M + \delta) + 2\delta}{M - \delta}.
\]
As $\delta$ and $\epsilon$ can be chosen arbitrarily small we get the upper bound.

\subsection{Lower bound}
We now prove the lower bound in Theorem~\ref{thm:deviations}. The only non trivial case is the one of a cocycle in unstable part $F^u_i$ or the central part $F^c$ of the Oseledets flag (\ref{eq:Oseledets_decomposition}). The proof is identical for both of them.

Let $X$ be a regular surface and $P$ be a transversal containing $X$ as in Section~\ref{subsection:transversal}. Let $Y \in P$ be recurrent for the Teichm\"uller flow and Oseledets generic. We denote by $t_n$ the sequence of return times in $P$. We fix a point $p \in Y$ with infinite future orbit and a cocycle $f \in F^u_i \backslash F^u_{i+1}$ in the unstable part of $H^1(X;\Z)$.

We use a similar decomposition as in the proof of Lemma~\ref{lem:cut_into_pieces}. For all $n$, we consider the sequence of return times of $p$ into $I^{(n)}$. For the $m$-th return time $T$ in $I^{(n)}$ we have a decomposition
\[
\gamma_T(p) = \gamma_-^{(1)}\ \gamma_-^{(2)}\ \ldots\ \gamma_-^{(n-1)}\ \zeta_{j_1}^{(n)} \zeta_{j_2}^{(n)} \ldots \zeta_{j_m}^{(n)}.
\]
where the sequence $(j_k) = (j_k(n))$ does only depend on $p$ and $n$ and the sequence $\gamma_-^{(k)}$ only on $p$. The length of the initial segment $\gamma_-^{(1)}\ \gamma_-^{(2)} \ldots\ \gamma_-^{(n-1)}$ corresponds to the first hitting time $T_n^-$ of $I^{(n)}$ starting from $p$.

Let $\varepsilon > 0$, we want to prove that the following holds
\[
\limsup_{T \to \infty} \frac{\log \left| \left\langle f, \gamma_T(p) \right\rangle \right| }{\log T} \geq \nu_i - \varepsilon.
\]
We are done if the above equality holds for infinitely many prefixes $\gamma_-^{(1)}\ \ldots\ \gamma_-^{(n-1)}$ of $\gamma_T(p)$. We assume that the latter condition does not hold and prove that the above inequality still holds. From that hypothesis, we get
\begin{equation} \label{eq:zorich_assumption_trick}
\log \left| \left\langle f, \gamma_{T_n^-}(p) \right\rangle \right| =
\log \left| \left\langle f, \gamma_-^{(1)} \gamma_-^{(2)} \ldots \gamma_-^{(n-1)} \right\rangle \right| 
\leq (\nu_i - \epsilon/2) \log T_n^-.
\end{equation}

\begin{lemma} \label{lem:goods_and_bads}
Let $K_1$ and $K_2$ be the constant of Lemma~\ref{lem:transversal}. There exists an index $\ell \in \{1,\ldots, \lfloor K_1 K_2 \rfloor \}$, a constant $C > 0$ and an infinite subset $N \subset \N$ such that for all $n \in N$
\[
\forall \ell' \in \{1,\ldots,\ell-1\},\ \frac{\left|\left\langle B^{(t_n)} f, \zeta_{j_{\ell'}(n)} \right\rangle \right|}{\left\| B^{(t_n)} f \right\|} \leq \frac{C}{K_1 K_2} \quad \text{and} \quad \frac{\left|\left\langle B^{(t_n)} f, \zeta_{j_\ell(n)} \right\rangle \right|}{\left\| B^{(t_n)} f \right\|} \geq C.
\]
\end{lemma}

\begin{proof}
From equivalence of norms on the finite dimensional vector space $H^1(X;\R)$, there exists a constant $C' > 0$ such that
\[
\forall v \in H^1(X;\Z), \quad \max_{j=1,\ldots,d}{|\langle v, \zeta_j\rangle|} \geq C' \|v\|.
\]

Now, since the length of each curve $\zeta_j$ is at least $K_1^{-1}$ and that after time $K_2$ all curves $\zeta_j$ appear (see Lemma~\ref{lem:transversal}), we know that before $\lfloor K_1 K_2 \rfloor $ return times in $I$ any geodesic $\gamma_T(p)$ passes through all rectangles. In particular for at least one of the curves $\zeta_{j_1(n)}$, \ldots $\zeta_{j_{K_1 K_2}(n)}$ we have a uniform lower bound on the pairing with $f$.

Now, consider the sequence of pieces in first position $\zeta_{j_1(n)}$ as $n \in \N$. If 
\[
\limsup_{n \to \infty} \frac{\left|\left\langle B^{(t_n)} f, \zeta_{j_1(n)} \right\rangle \right|}{\left\| B^{(t_n)} f \right\|} > 0
\]
then we are done by choosing $C$ to be the half of the $\limsup$ above. If not, we consider the sequence of pieces in second position $\zeta_{j_2(n)}$ and repeat the dichotomy. We know from the first part of the proof, that this process stops before the $(K_1 K_2)$-th position. We get a position $\ell$, a constant $C$, and a subsequence $N \subset \N$ that satisfy the right inequality of the statement of the lemma. By starting the subsequence far enough (i.e. considering $N \cap \{m,m+1,\ldots\}$ for $m$ big enough), by our construction since each of the $\liminf$ for $l'=1,\ldots,k-1$ is $0$, we may ensure by our construction that the $\ell-1$ inequalities on the left holds.
\end{proof}

Let $\ell$, $C$ and $N \subset \N$ that satisfies the conclusion of Lemma~\ref{lem:goods_and_bads}. Let $n \in N$ and $p_n \in I^{(n)}$ be the endpoint of the prefix $\gamma_-^{(1)}\ \ldots\ \gamma_-^{(n-1)}$. Then
\begin{align*}
\left\langle f, \gamma_{T_n}(p) \right\rangle
 &= \left\langle f, \gamma_{T^-_n}(p) \right\rangle + \left\langle f, \gamma_{T_n - T_n^-}(p_n) \right\rangle\\
 &= \left\langle f, \gamma_-^{(1)}\ \ldots\ \gamma_-^{(n-1)}\right\rangle 
 + \left\langle f, \zeta^{(n)}_{j_1(n)} \right\rangle + \ldots +  \left\langle f, \zeta^{(n)}_{j_{l-1}(n)} \right\rangle + \left\langle f, \zeta^{(n)}_{j_l(n)} \right\rangle \\
 &= \left\langle f, \gamma_-^{(1)}\ \ldots\ \gamma_-^{(n-1)}\right\rangle
 + \left\langle B^{(t_n)} f, \zeta_{j_1(n)} \right\rangle + \ldots + \left\langle B^{(t_n)} f, \zeta_{j_{l-1}(n)} \right\rangle + \left\langle B^{(t_n)} f, \zeta_{j_l(n)} \right\rangle.
\end{align*}
Using triangular inequality, we get
\begin{align*}
\left| \left\langle f, \gamma_{T_n}(p) \right\rangle \right| &
 \geq \left| \left\langle B^{(t_n)} f, \zeta_{j_l(n)} \right\rangle\right| - \sum_{k=1}^{l-1} \left| \left\langle B^{(t_n)} f, \zeta_{j_k(n)} \right\rangle \right| - \left| \left\langle f, \gamma_{T^-_n}(p) \right\rangle \right| \\
 & \geq C \left\| B^{(t_n)} f \right\| - \frac{(l-1) C}{K_1 K_2} \left\| B^{(t_n)} f \right\| - \left| \left\langle f, \gamma_{T_n^-}(p) \right\rangle \right| \\
 & \geq \frac{C}{K_1 K_2} \left\| B^{(t_n)} f \right\| - \left| \left\langle f, \gamma_{T^-_n}(p) \right\rangle \right|.
\end{align*}
We use twice Lemma~\ref{lem:transversal} to prove that $T_n$ grows like $e^{t_n}$. First of all, as $T_n$ is a time for which the orbit of $p$ under the linear flow has reached at least twice the interval $I^{(n)}$ we have $T_n > K_1^{-1} e^{t_n}$. Moreover, by construction, $T_n < K_1 (l+1) e^{t_n} <= K_1 (K_1 K_2 + 1) e^{t_n}$. Hence
\begin{equation} \label{eq:comp_tn_Tn}
\lim_{n \to \infty} \frac{\log T_n}{t_n} = 1.
\end{equation}
From our assumption (\ref{eq:zorich_assumption_trick}), for $n \in N$ big enough, the term $| \langle f, \gamma_{T_n^-} \rangle |$ is exponentially smaller than $\| B^{(t_n)} f\|$. We hence get that
\[
\limsup_{T \to \infty} \frac{|\langle f, \gamma_T(p) \rangle |}{\log T} \geq \limsup_{n \in N} \frac{\log \left(C/(K_1 K_2)\left\| B^{(t_n)} f \right\| - \left| \left\langle f, \gamma_{T_n^-}(p) \right\rangle \right| \right)}{t_n} = \limsup_{n \in N} \frac{\log \left\| B^{(t_n)} f \right\|}{\log t_n} = \nu_i.
\]
Where the last inequality holds, as we assume $Y$ to be Osseledets generic. In other words, under assumption~(\ref{eq:zorich_assumption_trick}) we exhibit a subsequence on which the $\limsup$ is achieved.

\subsection{Theorem~\ref{thm:divergence_rate_surface} for the whole locus $\Xablocus$}
We now prove our main theorem about the wind-tree model in the following form
\begin{lemma} \label{lem:deviations_for_almost_all_surfaces}
Let $0 < a < 1$ and $0 < b < 1$. Let $f \in H^1(\Xsurf(a,b);\Z^2)$ be the cocycle that defines the wind-tree model. For $Y \in U$, let $Y_\theta = e^{-i \theta} Y$. Then for all $Y \in U$, Lebesgue almost every $\theta \in S^1$, every point $p \in Y_\theta$ with infinite forward orbit for the linear flow
\[
\limsup_{T \to \infty} \frac{\log \left|\left\langle f, \gamma_T(p) \right\rangle \right|}{\log(T)} = \frac{2}{3}.
\]
\end{lemma}

\begin{proof}
By Chaika Eskin result~\cite{CE}, generic recurrence and Oseledets holds for almost every $\theta$. We only need to analyze in which components of the Oseledets splitting $f$ belongs to.

From Lemma~\ref{lem:vector_of_the_infinite_cover} and Lemma~\ref{lem:homology_splitting}, we know that the cocycle $f \in H^1(\Xsurf(a,b);\Z^2)$ decomposes into two pieces $f^{+-} \in E^{+-}(\Q)$ and $f^{-+} \in E^{-+}(\Q)$ where each of $E^{+-}(\R)$ and $E^{-+}(\R)$ are rank $2$ subbundles stable under the Kontsevich-Zorich cocycle. From Theorem~\ref{thm:value_of_lyap_exp} and Corollary~\ref{cor:strong_Oseledets}, the Lyapunov exponents of the KZ cocycle for $e^{i \theta} Y$ exists and are $2/3$ and $-2/3$ in both of $E^{+-}(\R)$ and $E^{-+}(\R)$. The only thing to prove in Lemma~\ref{lem:deviations_for_almost_all_surfaces} is that $f^{+-}$ (resp. $f^{-+}$) does not belong to the stable subspace of $E^{+-}(\R)$ (resp. $E^{-+}(\R)$) associated to $-2/3$). If $f$ belongs to the stable subspace, then $\|B^{(t)} f\|$ goes to zero as $t$ tends to infinity. But recall that the the Kontsevich-Zorich cocycle takes values in the set of integer matrices of determinant $1$. In particular it preserves the set $E^{+-}(\Z)$. On the other hand, $f$ is a rational vector and there exists $N$ such that $N f \in E^{+-}(\Z)$. In particular, the quantity $\|B^{(t)} f\|$ is bounded below by
\[
\frac{1}{N} \min_{v \in E^{+-}(\Z) \backslash \{0\}} \|v\|
\]
which is different from $0$ as the integer vectors $E^{+-}(\Z)$ form a lattice in $E^{+-}(\R)$.
\end{proof}

\end{document}